\documentclass[12pt,a4paper]{article}
\usepackage[utf8]{inputenc}
\usepackage{CJK}

\usepackage{aliascnt}
\usepackage{amsmath,amsthm}
\usepackage{amsfonts}
\usepackage{amssymb}
\usepackage{calligra}
\usepackage{calrsfs}
\usepackage{hyperref}
\usepackage{enumitem}
\usepackage[nameinlink,noabbrev,capitalise]{cleveref}
\usepackage{float}
\usepackage{afterpage}
\usepackage[style=numeric-comp]{biblatex}
\usepackage[left=2cm,right=2cm,top=2cm,bottom=2cm]{geometry}
\usepackage[mathscr]{euscript}
\usepackage{tikz-cd}
\usetikzlibrary{shapes,arrows}

\DeclareMathAlphabet{\mathcalligra}{T1}{calligra}{m}{n}

\theoremstyle{plain} 

\newtheorem{theorem}{Theorem}[section]

\newaliascnt{lemma}{theorem}
\newtheorem{lemma}[lemma]{Lemma}
\aliascntresetthe{lemma}

\newaliascnt{corollary}{theorem}

\aliascntresetthe{corollary}

\newaliascnt{proposition}{theorem}
\newtheorem{proposition}[proposition]{Proposition}
\aliascntresetthe{proposition}

\newaliascnt{remark}{theorem}
\newtheorem{remark}[remark]{Remark}
\aliascntresetthe{remark}

\newaliascnt{example}{theorem}

\aliascntresetthe{example}

\newaliascnt{conjecture}{theorem}
\newtheorem{conjecture}[conjecture]{Conjecture}
\aliascntresetthe{conjecture}

\newaliascnt{problem}{theorem}
\newtheorem{problem}[problem]{Problem}
\aliascntresetthe{problem}

\theoremstyle{definition} 
\newaliascnt{definition}{theorem}
\newtheorem{definition}[definition]{Definition}
\aliascntresetthe{definition}

\crefname{section}{Section}{Sections}
\crefname{definition}{Definition}{Definitions}
\crefname{theorem}{Theorem}{Theorems}
\crefname{lemma}{Lemma}{Lemmas}
\crefname{corollary}{Corollary}{Corollaries}
\crefname{proposition}{Proposition}{Propositions}
\crefname{remark}{Remark}{Remarks}
\crefname{example}{Example}{Examples}
\crefname{conjecture}{Conjecture}{Conjectures}
\crefname{question}{Question}{Questions}
\crefname{formula}{Formula}{Formulas}
\crefname{problem}{Problem}{Problems}

\numberwithin{equation}{section}

\newcommand{\K}{\mathbb{K}}
\newcommand{\A}{\mathcal{A}}

\crefname{enumi}{Case}{Cases}

\title{Generalizing Saito's Criterion for Nonfree Arrangements}

\author{Junyan CHU\thanks{Graduate School of Science, Kyoto University, Japan. \\ 
This research was partially supported by JST Moonshot R\&D Grant Number JPMJMS2021, and KAKENHI, Grant-in-Aid for Scientific Research (B) 25K00921 and (S) 25H00399.\\
 Email:chujy626@gmail.com \qquad
 ORCID: 0009-0006-7686-5766
}}

\begin{document}

\maketitle

\begin{abstract}
Saito's criterion is a foundational result that algebraically characterizes free hyperplane arrangements via the determinant of a square matrix of logarithmic derivations. It is natural to ask whether this criterion can be generalized to the non-free setting. To address this, we formulate a general problem concerning the maximal minors of a $p \times \ell$ ($p \geq \ell$) derivation matrix and the algebraic relations among their associated coefficients. Focusing on strictly plus-one generated (SPOG) arrangements, we completely solve this minor-based recognition problem under the assumption that $\operatorname{pd} D(\mathcal{A}) \leq 1$. As a direct consequence, we obtain a purely algebraic, necessary and sufficient characterization of SPOG arrangements in dimension three. Ultimately, this framework provides a computable bridge to post-free arrangement theory.
\end{abstract}


\section{Introduction}\label{sec:introduction}

Let $\mathcal A$ be a central hyperplane arrangement in a vector space
$V \cong \mathbb K^\ell$ over a field $\mathbb K$, and write
$S=\mathrm{Sym}(V^\ast)\cong \mathbb K[x_1,\dots,x_\ell]$.
The module of logarithmic derivations along $\mathcal A$
is the graded $S$-module
\[
D(\mathcal A)=\{\theta\in\mathrm{Der}_{\mathbb K}(S)=\bigoplus_{i=1}^\ell S\partial_{x_i}\mid \theta(\alpha_H)\in \alpha_H S
\text{ for all } H\in\mathcal A\},
\]
where $\partial_{x_i}=\frac{\partial}{\partial x_i}$ and $\alpha_H\in V^\ast$ is a defining linear form of $H$.
The arrangement $\mathcal A$ is called \emph{free} if $D(\mathcal A)$ is a free $S$-module. The module $D(\mathcal A)$ is a fundamental object bridging the combinatorics and algebra of arrangements (see, e.g., \cite{O-T}).
Recently, increasing attention has been directed toward \emph{nonfree} arrangements
\cite{spog,Abe2024TransAMS,AbeKawanoue2024,Chu2025CloseToFree}.

The fundamental tool for recognizing free arrangements is the classical Saito criterion (\Cref{saito}).
Beyond this classical result, several generalized Saito criteria exist in the literature, such as in the theory of holonomic divisors \cite{EpureSchulze2019} or projective settings \cite{FaenziJardimValles2024Generalized,FaenziJardimValles2024Avatars}. 
However, these generalizations maintain the goal of testing for freeness.

In this paper, we propose a different direction: generalizing Saito's criterion to test for non-free arrangements with higher homological complexity. A natural question is therefore the following:
\begin{center}
  \emph{Can one generalize Saito's criterion to characterize the generators of $D(\A)$ for non-free $\A$?}  
\end{center}
When $\A$ is nonfree, the module $D(\A)$ requires $p > \ell$ minimal generators. Consequently, the corresponding derivation matrix is of size $p \times \ell$, and the usual Saito determinant no longer exists. Our core idea is to replace this single determinant with the structured collection of maximal minors of the derivation matrix.
\begin{definition}\label{def-minors}
    Let $\theta_1,\dots,\theta_p \in D(\A)$ be $S$-independent, with $p\ge \ell$.
We define the matrix
\[
M[\theta_1,\ldots,\theta_p]
\]
to be the $p\times \ell$ matrix whose $(i,j)$-entry is $\theta_i(x_j)$, where $i=1,\dots,p$ and $j=1,\dots,\ell$.
For a subset
\[
I=\{i_1<\cdots<i_\ell\}\subseteq [p],
\]
let $M_I$ denote the $\ell\times \ell$ submatrix of $M[\theta_1,\ldots,\theta_p]$ consisting of the rows indexed by $I$, and define
\[
\Delta_I:=(-1)^{\sigma(I)}\det(M_I),
\qquad
\sigma(I):=\sum_{k=1}^{\ell}(i_k-k)
=\sum_{i\in I} i-\frac{\ell(\ell+1)}2.
\]
Equivalently, $(-1)^{\sigma(I)}$ is the sign of the shuffle permutation that moves the rows indexed by $I$ to the first $\ell$ positions while preserving their relative order.
\end{definition}
Because each $\theta_i \in D(\A)$, it is a standard fact that every minor $\Delta_I$ is divisible by $Q(\A)$. Hence, we can uniquely write $\Delta_I = g_I Q(\A)$ for some polynomial $g_I \in S$. 
In the free case ($p=\ell$), there is only one coefficient $g_I$, and Saito's criterion simply requires $g_I$ to be a nonzero constant. In our nonfree setting, this naturally evolves into the following fundamental problem.

\begin{problem}\label{question-saito}
Let $M$ be a $p \times \ell$ derivation matrix with maximal minors $\Delta_I = g_I Q(\A)$. Under what algebraic relations among the coefficients $\{g_I \mid I\subseteq [p],\ |I|=\ell\}$ can one conclude that $\theta_1,\dots,\theta_p$ minimally generate $D(\A)$?
\end{problem}

This problem is highly relevant to current research. For example, deciding whether the images of derivations generate the restriction $D(\A^H)$—a central difficulty in studying the Euler restriction map \cite{Abe2024TransAMS,AbeKawanoue2024}—requires an effective, computable test for generators.

To make \Cref{question-saito} tractable, we stratify nonfree arrangements by the projective dimension $\operatorname{pd} D(\A)$. The immediate step beyond freeness—projective dimension one—has recently become a highly active area of study, largely driven by Abe's work on \emph{strictly plus-one generated (SPOG)} arrangements \cite{spog,Abe2024TransAMS}. 
Motivated by this, we provide a complete answer to \Cref{question-saito} for the $\operatorname{pd} D(\A) = 1$ regime, specifically focusing on SPOG arrangements (see \Cref{spogarr}) where $p = \ell+1$.

Given $\ell+1$ homogeneous derivations $\theta_1,\dots,\theta_{\ell+1}\in D(\mathcal A)$, let $M=M[\theta_1,\ldots,\theta_{\ell+1}]$ be the associated $(\ell+1)\times \ell$ matrix. Note that for $I=[\ell+1]\setminus\{i\}$, we have $\Delta_I = (-1)^{\ell+1-i}\det(M_I)$. To simplify the notation, we write $Q=Q(\A),\, M_i:=M_I$ and define 
\[
\Delta_i := (-1)^i \det M_i = g_i Q
\]
for some $g_i \in S$.
Our main theorem is the following.
\begin{theorem}\label{Generalization-Saito-for-spog}
Assume $g_{\ell+1} \in S_1 \setminus \{0\}$, and that $g_1, \ldots, g_\ell \in S_{>0}$ have no non-trivial common divisor modulo $g_{\ell+1}$. 

If $\operatorname{pd} D(\mathcal{A}) \le 1$, then  $\A$ is SPOG. That is, $\theta_1, \dots, \theta_{\ell+1}$ form a minimal generating set for $D(\mathcal{A})$ with the unique relation
\[
    g_1\theta_1 + \cdots + g_{\ell+1}\theta_{\ell+1} = 0.
\]
\end{theorem}

For an arbitrary dimension $\ell$, the homological assumption $\operatorname{pd} D(\mathcal{A}) \le 1$ is necessary. However, for dimension $\ell = 3$, it is a standard consequence of the reflexivity of $D(\mathcal{A})$ that $\operatorname{pd} D(\mathcal{A}) \le \ell - 2 = 1$ always holds. Therefore, by applying \Cref{Generalization-Saito-for-spog}, we can completely remove the homological assumption, yielding a purely algebraic, necessary and sufficient characterization for 3-dimensional SPOG arrangements.

\begin{theorem}\label{Generalization-Saito-for-3dim-spog}
Assume $\ell = 3$. Then $\mathcal{A}$ is SPOG, and the set $\{\theta_1, \theta_2, \theta_3, \theta_4\}$ forms a minimal generating set for $D(\mathcal{A})$ satisfying the minimal degree relation 
\[
    g_1\theta_1 + g_2\theta_2 + g_3\theta_3 + g_4\theta_4 = 0,
\]
if and only if $g_4 \in S_1 \setminus \{0\}$, and the coefficients $g_1, g_2, g_3 \in S_{>0}$ have no non-trivial common divisor modulo $g_4$.
\end{theorem}

Conceptually, this minor-based extension of Saito's criterion replaces the single basis determinant of the free case with a structured family of maximal-minor invariants. This provides a genuinely computable, Saito-style recognition principle for modules of projective dimension one, complementing the homological and addition-deletion methods developed in \cite{spog,Abe2024TransAMS} and offering a new toolkit for studying restrictions of both free and non-free arrangements.

\noindent \textbf{Organization.}
The paper is organized as follows. 
\Cref{sec:preliminaries} recalls necessary definitions and preliminary results. 
\Cref{sec:proofs} is devoted to the proofs of our main results. 
Finally, \Cref{sec:conjectures} concludes the paper with a discussion of related conjectures for future research.

\vspace{5mm}
\noindent\emph{Acknowledgements.}
We would like to thank Takuro Abe and Shizuo Kaji for many helpful discussions.
 
 \section{ Preliminaries}\label{sec:preliminaries}

To simplify the discussion, we introduce the following algebraic definitions regarding the relations among derivations.

\begin{definition}
We say that a relation in $D(\A)$
\[
    f_1\theta_1 + \cdots + f_{\ell+1}\theta_{\ell+1} = 0
\]
is of \emph{minimal degree} (or \emph{primitive}) if the polynomials $f_1, \ldots, f_{\ell+1}$ have \textbf{no} non-trivial common divisor.
\end{definition}

\begin{definition}
We say that polynomials $f_1, \ldots, f_p \in S$ have a \emph{non-trivial common divisor $h$ modulo $f$} if there exists $h \in S_{>0}$ such that $f_i \in (h, f)$ for all $i$. 
\end{definition}

In \Cref{def-minors}, we introduced the maximal minors $\Delta_I$ of the $p \times \ell$ derivation matrix $M$. The most fundamental case occurs when $p = \ell$. In this scenario, there is only one maximal minor, which is simply the determinant of the square matrix $M$. The classical Saito's criterion states that freeness is completely characterized by this single determinant being a non-zero scalar multiple of the defining polynomial $Q=Q(\A)$.

\begin{theorem}[Saito's criterion \cite{Saito1980}]\label{saito}
Let $\theta_1, \ldots, \theta_{\ell} \in D(\A)$ be homogeneous derivations. Then the following conditions are equivalent:
\begin{enumerate}[label=(\arabic*)]
  \item $D(\A) = S \theta_1 \oplus \cdots \oplus S \theta_{\ell}$, i.e., $\A$ is free with basis $\{\theta_1, \ldots, \theta_{\ell}\}$.
  \item The derivations $\theta_1, \ldots, \theta_{\ell}$ are $S$-linearly independent, and 
  \[
    \det\big( M[ \theta_1,\ldots,\theta_\ell]\big) = c \, Q(\A)
  \]
  for some nonzero constant $c \in \K^*:=\K\setminus\{0\}$.
  \item The set $\{\theta_1,\ldots,\theta_{\ell}\}$ is $S$-independent and 
  \[
    \sum_{i=1}^{\ell}\deg \theta_i = |\A|.
  \]
\end{enumerate}
In this case, $\A$ is free with exponents $\exp(\A) = (\deg \theta_1, \ldots, \deg \theta_{\ell})$.
\end{theorem}

While free arrangements represent the ideal scenario where $\operatorname{pd} D(\mathcal{A}) = 0$, many natural geometric operations—such as deleting a hyperplane from a free arrangement—disrupt this freeness. To capture the algebraic structure of these ``slightly non-free'' arrangements, Abe \cite{spog} introduced the following class, which forms the most fundamental object in the $\operatorname{pd} D(\mathcal{A}) = 1$ regime.
\begin{definition}[Definition 1.1 in \cite{spog}]\label{spogarr}
An arrangement $\A$ is said to be \emph{strictly plus-one generated (SPOG)} if $D(\A)$ requires exactly $\ell+1$ minimal homogeneous generators $\theta_1,\ldots,\theta_{\ell+1}$, which satisfy a unique generating relation
    \[f_1\theta_1+\cdots+f_{\ell+1}\theta_{\ell+1}=0,\]
where $f_1,\ldots,f_{\ell+1}\in S$ with $f_{\ell+1}\in S_1\setminus\{0\}$.
\end{definition}

\begin{remark}
    Note that under this definition, the projective dimension is $\operatorname{pd} D(\A) = 1$, and $D(\A)$ admits the following minimal free resolution:
    \[
    0 \rightarrow S[-1-\deg \theta_{\ell+1}] \xrightarrow{\begin{pmatrix} f_1 & \cdots & f_{\ell} & f_{\ell+1}  \end{pmatrix}^T} \bigoplus_{k=1}^{\ell+1} S[-\deg \theta_k] \rightarrow D(\A) \rightarrow 0.
    \]
\end{remark}
To relate the degrees of the generators to the cardinality of $\A$, we utilize a well-known result regarding the graded Betti numbers of $D(\A)$. The following theorem is a simplified version of the general formula found in \cite{graded_betti2010}.

\begin{theorem}[Theorem 0.2 in \cite{graded_betti2010}]\label{betti-1}
If the logarithmic derivation module $D(\A)$ has a finite graded free resolution given by:
\begin{align*}
    0 \rightarrow \bigoplus_{i=1}^{r_k} S[-d_i^k] \rightarrow \cdots 
    \rightarrow \bigoplus_{i=1}^{r_1} S[-d_i^1] \rightarrow \bigoplus_{i=1}^{r_0} S[-d_i^0] \rightarrow D(\A) \rightarrow 0,
\end{align*}
then $|\A| = \sum_{j=0}^k (-1)^j \sum_{i=1}^{r_j} d_i^j$.      
\end{theorem}

By applying this homological degree formula to the $\operatorname{pd}=1$ case, we can rigorously formalize the degree constraints of the ``unique relation'' among the minimal generators.

\begin{proposition}\label{degree-relation-Arr}
Suppose that $\A$ is SPOG, and that $\{\theta_1,\ldots,\theta_{\ell+1}\}$ forms a minimal generating set for $D(\A)$ satisfying the unique relation
    \[f_1\theta_1+\cdots+f_{\ell+1}\theta_{\ell+1}=0.\]
Then, for any $i$ such that $f_i \neq 0$, the cardinality of $\A$ satisfies
\[
    |\A| = \sum_{j\neq i} \deg\theta_j - \deg f_i.
\]
\end{proposition}
\section{Proofs of Main Results}\label{sec:proofs}

Before presenting the main theorem, we first establish several auxiliary results.
\begin{lemma}\label{det-get-relation}
Let $\theta_1, \dots, \theta_{\ell+1}$ and $g_1, \dots, g_{\ell+1}$ be as defined above. Then we have the syzygy:
\[
g_1\theta_1 + \dots + g_{\ell+1}\theta_{\ell+1} = 0.
\]
\end{lemma}

\begin{proof}
Note that $M$ is the $(\ell+1) \times \ell$ matrix over $S$ with entries $M_{ij} = \theta_i(x_j)$. 
Consider the row vector of signed maximal minors $\Delta = (\Delta_1, \dots, \Delta_{\ell+1})$. 
The $j$-th entry of the product $\Delta M$ is given by the sum
\[
\sum_{i=1}^{\ell+1} \Delta_i M_{ij} = \sum_{i=1}^{\ell+1} (-1)^i \det(M_i) \theta_i(x_j).
\]
This expression corresponds to the Laplace expansion of the determinant of an $(\ell+1) \times (\ell+1)$ matrix obtained by duplicating the $j$-th column of $M$. Since any matrix with a repeated column has a determinant of zero, it follows that $\Delta M = 0$.

Recall from our notation that $\Delta_i = g_i Q$ for each $i=1, \dots, \ell+1$. Substituting this into the identity above, we obtain
\[
Q \cdot \sum_{i=1}^{\ell+1} g_i \theta_i(x_j) = 0
\]
for each $j=1, \dots, \ell$. Since $S$ is an integral domain and $Q \neq 0$, we must have $\sum_{i=1}^{\ell+1} g_i \theta_i(x_j) = 0$ for all $j=1, \dots, \ell$.

To show that this holds for any $f \in S$, recall that any derivation $\theta \in \mathrm{Der}_{\mathbb{K}}(S)$ is uniquely determined by its values on the coordinate variables. Specifically, for any $f \in S$, we have $\theta(f) = \sum_{j=1}^{\ell} \frac{\partial f}{\partial x_j} \theta(x_j)$. Thus,
\begin{align*}
\left( \sum_{i=1}^{\ell+1} g_i \theta_i \right)(f) &= \sum_{i=1}^{\ell+1} g_i \theta_i(f) \\
&= \sum_{i=1}^{\ell+1} g_i \left( \sum_{j=1}^{\ell} \frac{\partial f}{\partial x_j} \theta_i(x_j) \right) \\
&= \sum_{j=1}^{\ell} \frac{\partial f}{\partial x_j} \left( \sum_{i=1}^{\ell+1} g_i \theta_i(x_j) \right) = \sum_{j=1}^{\ell} \frac{\partial f}{\partial x_j} \cdot 0 = 0.
\end{align*}
This confirms that $\sum_{i=1}^{\ell+1} g_i \theta_i$ is the zero derivation, completing the proof.
\end{proof}

As a consequence of \Cref{det-get-relation}, we obtain the following property.

\begin{proposition}\label{relation-get-det}
If $g_{\ell+1} \neq 0$, then the following hold:
\begin{enumerate}[label=(\arabic*)]
    \item $g_{\ell+1}\eta \in S\theta_1 + \cdots + S\theta_{\ell}$ for any derivation $\eta \in D(\A)$.
    \item Suppose there is a relation of the form
    \[
    f_1\theta_1 + \cdots + f_{\ell}\theta_{\ell} + g_{\ell+1}\eta = 0
    \]
    for some $f_i \in S$. Then each coefficient $f_i$ is uniquely determined by the relation $\Gamma_i = f_i Q$, where $\Gamma_i$ is the signed maximal minor defined by
    \[
    \Gamma_i := (-1)^i\det M[\theta_1, \dots, \widehat{\theta_i}, \dots, \theta_{\ell}, \eta].
    \]
\end{enumerate}        
\end{proposition}

\begin{proof}
\begin{enumerate}[label=(\arabic*)]
    \item We can express the derivations $\theta_1, \dots, \theta_\ell$ in terms of the standard basis $\partial_{x_1}, \dots, \partial_{x_\ell}$ via the matrix equation
    \begin{align*}
        \begin{pmatrix}
        \theta_1 \\
        \vdots \\
        \theta_\ell
        \end{pmatrix}
        =
        M_{\ell+1}
        \begin{pmatrix}
        \partial_{x_1} \\
        \vdots \\
        \partial_{x_\ell}
        \end{pmatrix}.
    \end{align*}
    By multiplying both sides by the adjugate matrix of $M_{\ell+1}$, Cramer's rule implies that $(\det M_{\ell+1}) \partial_{x_j} \in S\theta_1 + \cdots + S\theta_\ell$ for all $j = 1, \dots, \ell$. 
    Recall that $\det M_{\ell+1} = (-1)^{\ell+1}\Delta_{\ell+1} = \pm g_{\ell+1} Q$. It follows that $g_{\ell+1} Q \partial_{x_j} \in S\theta_1 + \cdots + S\theta_{\ell}$.
    Consequently, for any $\eta \in D(\A)$, we can multiply by $g_{\ell+1} Q$ to find coefficients $h_i \in S$ such that
    \begin{equation}\label{eq:eta-relation}
        g_{\ell+1}Q\eta = h_1\theta_1 + \cdots + h_\ell \theta_\ell.
    \end{equation}
    
    We now show that $Q$ divides each $h_i$. Consider the signed minor $\Gamma_i$ defined in part (2). Since $\theta_1, \dots, \theta_\ell, \eta$ are all logarithmic derivations in $D(\A)$, the determinant of their coefficient matrix must be divisible by $Q$; hence $\Gamma_i \in Q S$. 
    Using equation \eqref{eq:eta-relation} and the linearity of the determinant in the last column, we compute:
    \begin{align*}
        g_{\ell+1}Q\Gamma_i &= (-1)^i g_{\ell+1}Q \det M[\theta_1, \dots, \widehat{\theta_i}, \dots, \theta_{\ell}, \eta] \\
        &= (-1)^i \det M[\theta_1, \dots, \widehat{\theta_i}, \dots, \theta_{\ell}, g_{\ell+1}Q\eta] \\
        &= (-1)^i \det M[\theta_1, \dots, \widehat{\theta_i}, \dots, \theta_{\ell}, h_i\theta_i].
    \end{align*}
    Moving the last column to the $i$-th position requires $\ell - i$ column swaps. Thus,
    \begin{align*}
        g_{\ell+1}Q\Gamma_i &= (-1)^i (-1)^{\ell-i} h_i \det M[\theta_1, \dots, \theta_i, \dots, \theta_\ell] \\
        &= (-1)^\ell h_i \det M_{\ell+1} \\
        &= (-1)^\ell h_i \big( (-1)^{\ell+1} g_{\ell+1} Q \big) = -h_i g_{\ell+1} Q.
    \end{align*}
    This implies $-h_i g_{\ell+1} Q \in g_{\ell+1} Q^2 S$. Since $S$ is an integral domain and $g_{\ell+1} Q \neq 0$, we conclude that $h_i \in QS$, meaning $Q \mid h_i$ for all $i = 1, \dots, \ell$. 
    Dividing equation \eqref{eq:eta-relation} by $Q$, we obtain $g_{\ell+1}\eta \in S\theta_1 + \cdots + S\theta_\ell$.

    \item Let $u_i = -h_i / Q \in S$. From the conclusion of part (1), we have a specific relation:
    \[
        u_1\theta_1 + \cdots + u_{\ell}\theta_{\ell} + g_{\ell+1}\eta = 0.
    \]
    Suppose we have another relation $f_1\theta_1 + \cdots + f_{\ell}\theta_{\ell} + g_{\ell+1}\eta = 0$. Subtracting the two yields
    \[
        (f_1 - u_1)\theta_1 + \cdots + (f_{\ell} - u_\ell)\theta_{\ell} = 0.
    \]
    Since $g_{\ell+1} \neq 0$, we have $\det M_{\ell+1} \neq 0$, which means the derivations $\theta_1, \dots, \theta_\ell$ are linearly independent over $S$. Thus, $f_i = u_i$ for all $i = 1, \dots, \ell$. 
    Finally, from our calculation in part (1), we established $g_{\ell+1}Q\Gamma_i = -h_i g_{\ell+1} Q$. Canceling $g_{\ell+1}Q$ gives $\Gamma_i = -h_i$. Therefore, $f_i = u_i = -h_i / Q = \Gamma_i / Q$, which proves $\Gamma_i = f_i Q$.
\end{enumerate}
\end{proof}

\begin{proposition}\label{sufficient-generalization-saito-criterion}
If $\theta_1, \ldots, \theta_{\ell+1}$ form a minimal generating set for $D(\mathcal{A})$ and satisfy a unique relation
\begin{align}\label{equFtheta}
    f_1\theta_1 + \cdots + f_{\ell+1}\theta_{\ell+1} = 0 
\end{align}
of minimal degree, then there exists $c \in \mathbb{K}^*$ such that $\Delta_i = c f_i Q$ for all $i = 1, \ldots, \ell+1$.

Moreover, if $f_{\ell+1} \neq 0$, then $f_1, \ldots, f_\ell$ have no non-trivial common divisor modulo $f_{\ell+1}$.
\end{proposition}

\begin{proof}
By \Cref{det-get-relation}, we have the syzygy
\[
g_1\theta_1 + \cdots + g_{\ell+1}\theta_{\ell+1} = 0,
\]
where $\Delta_i = g_i Q$. Since the relation \eqref{equFtheta} is unique and of minimal degree, any other syzygy must be a multiple of it. Thus, there exists some $c \in S$ such that $g_i = c f_i$ for all $i$. 
Noting that $\Delta_i$ is the determinant of the matrix obtained by omitting the derivation $\theta_i$, we have $\deg(\Delta_i) = \sum_{j \neq i} \deg(\theta_j)$. Since $\Delta_i = g_i Q$ and $\deg(Q) = |\mathcal{A}|$, it follows that
\[
\deg(g_i) = \sum_{j \neq i} \deg(\theta_j) - |\mathcal{A}|.
\]
On the other hand, by \Cref{degree-relation-Arr}, the minimal degree relation satisfies $\deg(f_i) = \sum_{j \neq i} \deg(\theta_j) - |\mathcal{A}|$. 
Therefore, $\deg(g_i) = \deg(f_i)$, which implies $\deg(c) = 0$. Since not all $g_i$ are zero, we conclude that $c$ is a non-zero constant, i.e., $c \in \mathbb{K}^*$. This proves the first assertion.

For the second assertion, we proceed by contradiction. Suppose $f_1, \ldots, f_\ell$ share a non-trivial common divisor $h \in S_{>0}$ modulo $f_{\ell+1}$. Then we can write
\[
f_i = h f_i' + f_{\ell+1} k_i
\]
for some $f_i', k_i \in S$ and for all $i = 1, \dots, \ell$. 
Since the relation \eqref{equFtheta} is of minimal degree, the coefficients $f_1, \dots, f_{\ell+1}$ cannot share a common divisor across all terms, implying that $\gcd(h, f_{\ell+1}) = 1$.

Define a new derivation 
\[
\theta = \theta_{\ell+1} + \sum_{i=1}^{\ell} k_i \theta_i \in D(\mathcal{A}).
\]
Substituting $f_i$ into \Cref{equFtheta}, we obtain
\begin{align}\label{eq:ftheta0}
    h \sum_{i=1}^\ell f_i' \theta_i + f_{\ell+1} \theta = 0,
\end{align}
which implies $f_{\ell+1} \theta = -h \sum_{i=1}^\ell f_i' \theta_i$. 

Since $S$ is a unique factorization domain and $\gcd(h, f_{\ell+1}) = 1$, it follows that $h$ must divide the derivation $\theta$. Let $\theta = h \theta'$ for some derivation $\theta'$. 
Consequently, $h \theta' = \theta \in D(\mathcal{A})$ and $f_{\ell+1} \theta' = -\sum_{i=1}^\ell f_i' \theta_i \in D(\mathcal{A})$. 
Applying these to the defining polynomial $Q$, we have $h \theta'(Q) \in QS$ and $f_{\ell+1} \theta'(Q) \in QS$. Because $h$ and $f_{\ell+1}$ are coprime, it follows that $\theta'(Q) \in QS$, which guarantees $\theta' \in D(\mathcal{A})$.

Finally, substituting $\theta = h \theta'$ back into our definition of $\theta$ yields
\[
\theta_{\ell+1} = h \theta' - \sum_{i=1}^{\ell} k_i \theta_i.
\]
Since $h \in S_{>0}$, we have $\deg(\theta') < \deg(h \theta') = \deg(\theta) = \deg(\theta_{\ell+1})$. 
This shows that $\theta_{\ell+1}$ can be generated by elements of strictly lower degree alongside the other generators $\theta_1, \dots, \theta_\ell$. This contradicts the assumption that $\theta_1, \dots, \theta_{\ell+1}$ forms a minimal generating set for $D(\mathcal{A})$. 
\end{proof}

The following lemma only involves \(\theta_1, \dots, \theta_{\ell}\), but we retain the same notation for consistency.

\begin{lemma}\label{det-l+1-free-basis}
   Assume \(g_{\ell+1} \in S_1 \setminus \{0\}\).
   If $\A$ is free, there exists \(k\in \{1,\ldots,\ell\}\) and \(\eta\in D(\A)\) such that \(g_{\ell+1}\eta\in \K^*\theta_k+\sum_{j\neq k, j\leq \ell}S\theta_j\) and the set
   \[\theta_1,\ldots,\theta_{k-1},\eta, \theta_{k+1},\ldots,\theta_{\ell}\]
   form a basis of \(D(\A)\).
\end{lemma}

\begin{proof}
Since $\mathcal{A}$ is free, we can choose a basis $\{\eta_1, \ldots, \eta_\ell\}$ for $D(\mathcal{A})$. Let $M(\eta)=M[\eta_1, \ldots, \eta_\ell]$ be the $\ell \times \ell$ coefficient matrix of these derivations. By scaling one of the basis elements by a non-zero constant if necessary, we may assume $\det M(\eta) = Q$.

Let $U \in S^{\ell \times \ell}$ be the transition matrix such that
\[
    (\theta_1, \dots, \theta_\ell) = (\eta_1, \dots, \eta_\ell) U.
\]
Taking the determinant of the corresponding coefficient matrices, we obtain 
$$\det M_{\ell+1} = \det M(\eta) \det(U).$$
Recall from our notation that $\det M_{\ell+1} = (-1)^{\ell+1} g_{\ell+1} Q$. Substituting $\det M(\eta) = Q$, we get:
\[
    (-1)^{\ell+1} g_{\ell+1} = \det U.
\]

The determinant of $U$ is given by the sum of products of its entries. Since $\det(U) = (-1)^{\ell+1} g_{\ell+1} \in S_1 \setminus \{0\}$ has degree 1, and the entries of $U$ have non-negative degrees, at least one term in the Leibniz expansion of the determinant must consist of exactly one entry of degree 1 and $\ell-1$ entries of degree 0. The product of these $\ell-1$ constants corresponds to a term in the expansion of some $(\ell-1) \times (\ell-1)$ minor of $U$. Consequently, the classical adjugate matrix $\operatorname{adj}(U)$ must contain at least one entry in $\mathbb{K}^*$.

By appropriately reordering the indices of the sets $\{\theta_j\}$ and $\{\eta_j\}$, we may assume without loss of generality that the diagonal entry $\operatorname{adj}(U)_{\ell,\ell} \in \mathbb{K}^*$. We claim that setting $\eta := \eta_\ell$ and $k := \ell$ satisfies the lemma.

First, multiplying the transition equation by $\operatorname{adj}(U)$ from the right yields:
\[
    (\theta_1, \dots, \theta_\ell) \operatorname{adj}(U) = (\eta_1, \dots, \eta_\ell) \det(U) = (-1)^{\ell+1} g_{\ell+1} (\eta_1, \dots, \eta_\ell).
\]
Comparing the $\ell$-th components on both sides, we obtain:
\[
    (-1)^{\ell+1} g_{\ell+1} \eta_\ell = \sum_{j=1}^{\ell} \operatorname{adj}(U)_{j,\ell} \theta_j.
\]
Since the coefficient of $\theta_\ell$ is $\operatorname{adj}(U)_{\ell,\ell} \in \mathbb{K}^*$, it immediately follows that $g_{\ell+1} \eta_\ell \in \mathbb{K}^* \theta_\ell + \sum_{j < \ell} S \theta_j$.

Finally, let $U_{\ell,\ell}$ denote the submatrix obtained by removing the $\ell$-th row and $\ell$-th column of $U$. Modulo $\eta_\ell$, the transition equation becomes:
\[
    (\theta_1, \dots, \theta_{\ell-1}) \equiv (\eta_1, \dots, \eta_{\ell-1}) U_{\ell,\ell} \pmod{\eta_\ell}.
\]
Since $\det(U_{\ell,\ell}) = \operatorname{adj}(U)_{\ell,\ell} \in \mathbb{K}^*$, the matrix $U_{\ell,\ell}$ is invertible over $S$. Therefore, $\{\theta_1, \dots, \theta_{\ell-1}\}$ freely generates the quotient module $D(\mathcal{A}) / S\eta_\ell$, which implies that $\{\theta_1, \dots, \theta_{\ell-1}, \eta_\ell\}$ forms a basis for $D(\mathcal{A})$.
\end{proof}

If some $g_i \in \mathbb{K}^*$, then $\mathcal{A}$ is free by \Cref{saito}. We now consider the case where the coefficients are polynomials of positive degree.

\begin{lemma} \label{NoCommondivisorCauseNotfree}
Assume $g_{\ell+1} \in S_1 \setminus \{0\}$, and that $g_1, \ldots, g_\ell \in S_{>0}$ have no non-trivial common divisor modulo $g_{\ell+1}$. Then $\mathcal{A}$ is not free.
\end{lemma}

\begin{proof}
Suppose, for the sake of contradiction, that $\mathcal{A}$ is free. 
By \Cref{det-l+1-free-basis}, after possibly reordering the generators and scaling by a non-zero constant, we may assume there exists a derivation $\eta \in D(\mathcal{A})$ such that 
\[
    g_{\ell+1}\eta = \theta_\ell + \sum_{j<\ell} f_j \theta_j
\]
for some $f_j \in S$, and that the set $\{\theta_1, \ldots, \theta_{\ell-1}, \eta\}$ forms a basis for $D(\mathcal{A})$.

Since this set is a basis, we can express the generator $\theta_{\ell+1} \in D(\mathcal{A})$ as a linear combination of its elements:
\[
    \theta_{\ell+1} = f' \eta + \sum_{j<\ell} f_j' \theta_j
\]
for some coefficients $f', f_j' \in S$. 

Recall from \Cref{det-get-relation} that we have the syzygy:
\begin{align}\label{sumgl=0}
    g_1\theta_1 + \cdots + g_\ell\theta_\ell + g_{\ell+1}\theta_{\ell+1} = 0.
\end{align}
From our first equation, we can write $\theta_\ell = g_{\ell+1}\eta - \sum_{j<\ell} f_j \theta_j$. Substituting this and the expression for $\theta_{\ell+1}$ into the syzygy \eqref{sumgl=0}, we obtain:
\[
    \sum_{j<\ell} g_j \theta_j + g_\ell \left( g_{\ell+1}\eta - \sum_{j<\ell} f_j \theta_j \right) + g_{\ell+1} \left( f' \eta + \sum_{j<\ell} f_j' \theta_j \right) = 0.
\]
Rearranging the terms by grouping the basis elements yields:
\[
    \sum_{j<\ell} (g_j - g_\ell f_j + g_{\ell+1} f_j') \theta_j + g_{\ell+1}(g_\ell + f') \eta = 0.
\]
Since the elements $\theta_1, \ldots, \theta_{\ell-1}, \eta$ form a basis, they are linearly independent over $S$. Therefore, all coefficients must vanish. In particular, for $j = 1, \ldots, \ell - 1$, we have:
\[
    g_j - g_\ell f_j + g_{\ell+1} f_j' = 0 \implies g_j = g_\ell f_j - g_{\ell+1} f_j'.
\]
This equation implies that $g_j \equiv g_\ell f_j \pmod{g_{\ell+1}}$ for all $j < \ell$. Thus, $g_\ell$ divides every $g_j$ modulo $g_{\ell+1}$. 

Since $g_\ell \in S_{>0}$ by assumption, $g_\ell$ itself is a non-trivial common divisor of the set $\{g_1, \ldots, g_\ell\}$ modulo $g_{\ell+1}$. This directly contradicts the hypothesis that $g_1, \ldots, g_\ell$ have no non-trivial common divisor modulo $g_{\ell+1}$. Thus, $\mathcal{A}$ cannot be free.
\end{proof}

\begin{proof}[Proof of \Cref{Generalization-Saito-for-spog}]
By \Cref{NoCommondivisorCauseNotfree}, $\mathcal{A}$ is not free, which implies $\operatorname{pd} D(\mathcal{A}) \ge 1$. Since we assumed $\operatorname{pd} D(\mathcal{A}) \le 1$, we must have $\operatorname{pd} D(\mathcal{A}) = 1$.

Extend the $S$-independent set $\{ \theta_1, \ldots, \theta_\ell \}$ to a generating set 
\[
    G = \{\theta_1, \ldots, \theta_\ell, \eta_{\ell+1}, \ldots, \eta_p\}
\]
of $D(\mathcal{A})$ such that $\eta_j \notin S(G \setminus \{\eta_j\})$ for all $j = \ell+1, \ldots, p$. 

First, we prove that $\theta_i \notin S(G \setminus \{\theta_i\})$ for all $i = 1, \ldots, \ell$, so that $G$ is indeed a minimal generating set. Assume for contradiction that there exists some $i \in \{1, \ldots, \ell\}$ such that
\begin{equation}\label{eq:theta_i-a}
    \theta_i = \sum_{\substack{k=1 \\ k \neq i}}^\ell p_k \theta_k + \sum_{j=\ell+1}^p h_j \eta_j, \quad \text{with } p_k \in S \text{ and } h_j \in S_{>0}.
\end{equation}
By \Cref{relation-get-det}, for each $j = \ell+1, \ldots, p$, the derivations $\theta_1, \dots, \theta_\ell, \eta_j$ satisfy a relation
\begin{equation}\label{eq:rel_theta_j-a}
    g_{\ell+1}\eta_j = g_1^j\theta_1 + \cdots + g_\ell^j\theta_\ell,
\end{equation}
where $g_k^j Q = (-1)^k \det M[\theta_1, \dots, \widehat{\theta_k}, \dots, \theta_\ell, \eta_j]$. 
Multiplying \eqref{eq:theta_i-a} by $g_{\ell+1}$ and substituting \eqref{eq:rel_theta_j-a} into the right-hand side, we can express $g_{\ell+1}\theta_i$ entirely in terms of $\{\theta_1, \dots, \theta_\ell\}$. Since $\theta_1, \ldots, \theta_\ell$ are $S$-independent, we can equate the coefficients of $\theta_i$ on both sides to obtain:
\[
    g_{\ell+1} = \sum_{j=\ell+1}^p h_j g_i^j.
\]
Since $g_{\ell+1} \in S_1 \setminus \{0\}$ has degree 1 and each $h_j \in S_{>0}$ has degree $\ge 1$, there must exist some $j$ such that $h_j\neq 0$ has degree 1 and $g_i^j \in \mathbb{K}^*$. However, if $g_i^j \in \mathbb{K}^*$, \Cref{saito} would imply that the set $\{\theta_1, \dots, \widehat{\theta_i}, \dots, \theta_\ell, \eta_j\}$ forms a free basis for $D(\mathcal{A})$, contradicting the fact that $\mathcal{A}$ is not free. Thus, $G$ must be a minimal set of generators.

Next, we show that $p = \ell+1$. Since $\operatorname{pd} D(\mathcal{A}) = 1$ and the rank of $D(\mathcal{A})$ is $\ell$, the syzygy module must have rank $p - \ell$. This yields a minimal free resolution:
\begin{equation}\label{eq:minres-a}
    0 \rightarrow \bigoplus_{j=\ell + 1}^{p} S[-e_j] \xrightarrow{R} \bigoplus_{i=1}^{p} S[-d_i] \rightarrow D(\mathcal{A}) \rightarrow 0,
\end{equation}
where the $d_i$'s are the degrees of the minimal generators in $G$, and the $e_j$'s are the degrees of the minimal relations. Because $\{\theta_1, \dots, \theta_\ell\}$ is $S$-independent, every minimal relation must involve some $\eta_j$. By reordering if necessary, we can pair them such that $e_j > \deg \eta_j$ for all $j = \ell+1, \ldots, p$.
Since $\Delta_{\ell+1}=(-1)^{\ell+1}\det M[\theta_1,\ldots,\theta_\ell]=g_{\ell+1}Q$, we have $\sum_{i=1}^\ell \deg \theta_i=|\A|+1$.
By \Cref{degree-relation-Arr}, we have
\[
    |\mathcal{A}| = \sum_{i=1}^\ell \deg \theta_i + \sum_{j=\ell+1}^p \deg \eta_j - \sum_{j=\ell+1}^p e_j = |\mathcal{A}| + 1 - \sum_{j=\ell+1}^p (e_j - \deg \eta_j).
\]
This simplifies to $\sum_{j=\ell+1}^p (e_j - \deg \eta_j) = 1$. Since $e_j - \deg \eta_j \ge 1$ for each minimal relation, there can be exactly one relation. Consequently, $p - \ell = 1$, yielding $p = \ell+1$. The unique minimal relation up to scaling is \eqref{eq:rel_theta_j-a} with $j = \ell+1$.

Finally, we claim that the original set $\{\theta_1, \ldots, \theta_\ell, \theta_{\ell+1}\}$ forms this minimal generating set for $D(\mathcal{A})$. By \Cref{det-get-relation}, we have the given syzygy:
\begin{align}\label{eq:givenSyzygy}
       g_1\theta_1 + \cdots + g_\ell\theta_\ell + g_{\ell+1}\theta_{\ell+1} = 0. 
\end{align}
If $\theta_{\ell+1} \in S\theta_1 + \cdots + S\theta_\ell + \mathbb{K}^*\eta_{\ell+1}$, we can simply replace $\eta_{\ell+1}$ with $\theta_{\ell+1}$, and we are done. Otherwise, since $G$ is a basis for the generators, we can write:
\[
    \theta_{\ell+1} = u_1 \theta_1 + \cdots + u_\ell \theta_\ell + u \eta_{\ell+1}, \quad \text{with } u \in S_{>0}.
\]
Substituting this into \Cref{eq:givenSyzygy}, we obtain:
\[
    \sum_{i=1}^\ell (g_i + g_{\ell+1} u_i)\theta_i + g_{\ell+1} u \eta_{\ell+1} = 0.
\]
Because the syzygy module is generated by the single relation \eqref{eq:rel_theta_j-a}, this equation must be a polynomial multiple of $g_{\ell+1}\eta_{\ell+1} - \sum g_i^{\ell+1}\theta_i = 0$. In particular, the coefficient of $\eta_{\ell+1}$ dictates that the multiplier is exactly $u$. Thus, comparing the coefficients of $\theta_i$, we get:
\[
    g_i + g_{\ell+1} u_i = u g_i^{\ell+1}
\]
for each $i = 1, \ldots, \ell$. This implies that:
\[
    g_i \equiv u g_i^{\ell+1} \pmod{g_{\ell+1}} \quad \text{for } i = 1, \ldots, \ell.
\]
Since $u \in S_{>0}$, it acts as a non-trivial common divisor of $g_1, \ldots, g_\ell$ modulo $g_{\ell+1}$. This perfectly contradicts our initial assumption. Thus, $u$ must belong to $\mathbb{K}^*$, confirming that $\{\theta_1, \ldots, \theta_{\ell+1}\}$ minimally generates $D(\mathcal{A})$.
\end{proof}

\section{Future Work}\label{sec:conjectures}

We outline several directions related to this work that we believe are both promising and feasible:

\begin{enumerate}[label=(\arabic*)]
    \item We consider two possible extensions of \Cref{Generalization-Saito-for-spog}.

    First, completely dropping the homological assumption $\operatorname{pd} D(\A) \le 1$:
    \begin{conjecture}\label{conj:SPOG_characterisation}
    $\A$ is SPOG, and the set $\{\theta_1,\dots,\theta_{\ell+1}\}$ forms a minimal generating set for $D(\A)$ satisfying the unique minimal degree relation
    \[
        g_1\theta_1+\cdots+g_{\ell+1}\theta_{\ell+1}=0,
    \]
    if and only if $g_{\ell+1}\in S_1\setminus\{0\}$, and the coefficients $g_1, \ldots, g_\ell\in S_{>0}$ have no non-trivial common divisor modulo $g_{\ell+1}$.
    \end{conjecture}

    Second, dropping the condition $g_{\ell+1}\in S_1$ leads to a more general Betti number prediction:
    \begin{conjecture}\label{conj:freeres}
    When $\operatorname{pd} D(\A)=1$, the module $D(\A)$ admits a minimal free resolution of the following form:
    \[
        0 \rightarrow \bigoplus_{i=\ell + 1}^{p} S[-d_i - 1] \rightarrow \bigoplus_{i=1}^{p} S[-d_i] \rightarrow D(\A) \rightarrow 0.
    \]
    \end{conjecture}

    \item More generally, it is natural to ask how one can characterize the condition under which a subset $G = \{\theta_i \in D(\A) \mid i \in J\}$ forms a set of (minimal) generators for $D(\A)$, expressed purely in terms of the ideal $J(G)$. Here, $J(G)$ is generated by $\det M_I$, where $I \subseteq J$ with $|I| = \ell$.

    For example, at the opposite extreme from the free case (the generic case), we conjecture the following:
    \begin{conjecture}
    Let $G=\{\theta_i\in D(\A) \mid i \in J\}$. The ideal $J(G)$ satisfies $J(G) = S_{\geq k} \cdot Q$, where
    \[
        k = (\ell - 1)(|\A| - \ell - 1),
    \]
    if and only if $\A$ is generic and $G$ forms a minimal generating set of $D(\A)$.
    \end{conjecture}

    \item Finally, we consider the projective dimension of the restriction, directly relating to Orlik's Conjecture:
    \begin{conjecture}
    Let $H\in \A$, and suppose $\A$ is free. Then $\operatorname{pd} D(\A^H)\leq 1$.
    \end{conjecture}
    Notably, this property has been observed in all known counterexamples to Orlik's conjecture.
    Building upon this, we raise the following algebraic problem regarding the deletion-restriction behavior in the $\operatorname{pd}=1$ regime:
    \begin{problem}
    Let $H,L\in \A$, and suppose $\A$ is free. If the deletion $\A\setminus\{L\}$ is SPOG, what is the explicit structure of the restriction $D(\A\setminus\{L\})^H$ and its multiarrangement counterpart $D((\A\setminus\{L\})^H,m)$?
    \end{problem}

The combination of the Euler restriction and our conjectural insights is poised to provide transformative analytical tools for the study of $D(\A^H)$ beyond the free regime.
\end{enumerate}

We believe that pursuing these directions will significantly deepen our understanding of the algebraic and geometric structures governing hyperplane arrangements.

\printbibliography

@article {spog,
    AUTHOR = {Abe, Takuro},
     TITLE = {Plus-one generated and next to free arrangements of
              hyperplanes},
    journal = {International Mathematics Research Notices},
    volume = {2021},
    number = {12},
    pages = {9233-9261},
    year = {2019},
    month = {06},
      ISSN = {1073-7928,1687-0247},
   MRCLASS = {52C35 (32S22)},
  MRNUMBER = {4276319},
MRREVIEWER = {Piotr\ Pokora},
       DOI = {10.1093/imrn/rnz099},
       URL = {https://doi.org/10.1093/imrn/rnz099},
}

@book {O-T,
    AUTHOR = {Orlik, Peter and Terao, Hiroaki},
     TITLE = {Arrangements of hyperplanes},
    SERIES = {Grundlehren der mathematischen Wissenschaften [Fundamental
              Principles of Mathematical Sciences]},
    VOLUME = {300},
 PUBLISHER = {Springer-Verlag, Berlin},
      YEAR = {1992},
     PAGES = {xviii+325},
      ISBN = {3-540-55259-6},
   MRCLASS = {52B30 (14F35 20F36 20F55 32S25 57N65)},
  MRNUMBER = {1217488},
MRREVIEWER = {Michel\ Yves\ Jambu},
       DOI = {10.1007/978-3-662-02772-1},
       URL = {https://doi.org/10.1007/978-3-662-02772-1},
}

@article {Saito1980,
    AUTHOR = {Saito, Kyoji},
     TITLE = {Theory of logarithmic differential forms and logarithmic vector fields},
   JOURNAL = {J. Fac. Sci. Univ. Tokyo Sect. IA Math.},
  FJOURNAL = {Journal of the Faculty of Science. University of Tokyo.
              Section IA. Mathematics},
    VOLUME = {27},
      YEAR = {1980},
    NUMBER = {2},
     PAGES = {265--291},
      ISSN = {0040-8980},
   MRCLASS = {32G11 (14D05 32B30)},
  MRNUMBER = {586450},
MRREVIEWER = {Zoghman\ Mebkhout},
}

@article {graded_betti2010,
    AUTHOR = {Marco-Buzun\'{a}riz, M. A. and Mart\'{i}n-Morales, J.},
     TITLE = {Graded {B}etti numbers of the logarithmic derivation module},
   JOURNAL = {Comm. Algebra},
  FJOURNAL = {Communications in Algebra},
    VOLUME = {38},
      YEAR = {2010},
    NUMBER = {11},
     PAGES = {4348--4361},
      ISSN = {0092-7872,1532-4125},
   MRCLASS = {13N15},
  MRNUMBER = {2824799},
       DOI = {10.1080/00927870903366918},
       URL = {https://doi.org/10.1080/00927870903366918},
}

@article{Abe2024TransAMS,
    AUTHOR = {Abe, Takuro},
     TITLE = {Projective dimensions of hyperplane arrangements},
   JOURNAL = {Trans. Amer. Math. Soc.},
  FJOURNAL = {Transactions of the American Mathematical Society},
    VOLUME = {377},
      YEAR = {2024},
    NUMBER = {11},
     PAGES = {7793--7827},
      ISSN = {0002-9947,1088-6850},
   MRCLASS = {32S22 (52C35)},
  MRNUMBER = {4806197},
MRREVIEWER = {Piotr\ Pokora},
       DOI = {10.1090/tran/9196},
       URL = {https://doi-org.kyoto-u.idm.oclc.org/10.1090/tran/9196},
}

@article{EpureSchulze2019,
    AUTHOR = {Epure, Raul and Schulze, Mathias},
     TITLE = {A {S}aito criterion for holonomic divisors},
   JOURNAL = {Manuscripta Math.},
  FJOURNAL = {Manuscripta Mathematica},
    VOLUME = {160},
      YEAR = {2019},
    NUMBER = {1-2},
     PAGES = {1--8},
      ISSN = {0025-2611,1432-1785},
   MRCLASS = {32S65 (13H10 13N15)},
  MRNUMBER = {3983384},
       DOI = {10.1007/s00229-018-1065-5},
       URL = {https://doi-org.kyoto-u.idm.oclc.org/10.1007/s00229-018-1065-5},
}

@article{FaenziJardimValles2024Avatars,
    AUTHOR = {Faenzi, Daniele and Jardim, Marcos and Vall\`es, Jean},
     TITLE = {Saito criterion and its avatars},
   JOURNAL = {Rend. Circ. Mat. Palermo (2)},
  FJOURNAL = {Rendiconti del Circolo Matematico di Palermo. Second Series},
    VOLUME = {73},
      YEAR = {2024},
    NUMBER = {6},
     PAGES = {2233--2243},
      ISSN = {0009-725X,1973-4409},
   MRCLASS = {13N15 (14J60 14M10 32S65)},
  MRNUMBER = {4808374},
MRREVIEWER = {Amit\ Patra},
       DOI = {10.1007/s12215-024-01061-z},
       URL = {https://doi-org.kyoto-u.idm.oclc.org/10.1007/s12215-024-01061-z},
}

@misc{FaenziJardimValles2024Generalized,
      title={A generalized Saito freeness criterion}, 
      author={Daniele Faenzi and Marcos Jardim and Jean Vallès},
      year={2024},
      eprint={2407.14082},
      archivePrefix={arXiv},
      primaryClass={math.AC},
      url={https://arxiv.org/abs/2407.14082}, 
}

@article {Chu2025CloseToFree,
    AUTHOR = {Chu, Junyan},
     TITLE = {Free resolution of the logarithmic derivation modules of close
              to free arrangements},
   JOURNAL = {J. Algebraic Combin.},
  FJOURNAL = {Journal of Algebraic Combinatorics. An International Journal},
    VOLUME = {61},
      YEAR = {2025},
    NUMBER = {2},
     PAGES = {Paper No. 26, 27},
      ISSN = {0925-9899,1572-9192},
   MRCLASS = {52C35 (14N20)},
  MRNUMBER = {4873880},
MRREVIEWER = {Piotr\ Pokora},
       DOI = {10.1007/s10801-025-01394-7},
       URL = {https://doi-org.kyoto-u.idm.oclc.org/10.1007/s10801-025-01394-7},
}

@misc{AbeKawanoue2024,
      title={Cokernels of the Euler restriction map of logarithmic derivation modules}, 
      author={Takuro Abe and Hiraku Kawanoue},
      year={2024},
      eprint={2406.00305},
      archivePrefix={arXiv},
      primaryClass={math.CO},
      url={https://arxiv.org/abs/2406.00305}, 
}

\end{document}